\documentclass[a4paper, 12pt, reqno]{amsart}
\usepackage{amssymb}
\usepackage{amsthm}
\usepackage{amsmath}
\usepackage{amscd} 
\usepackage{mathtools}
\usepackage{enumerate,enumitem}
\usepackage{a4wide}
\usepackage{microtype}
\usepackage{listings}
\lstdefinelanguage{Sage}[]{Python}
{morekeywords={False,sage,True},sensitive=true}
\usepackage[colorlinks=true,allcolors=blue]{hyperref}

\numberwithin{equation}{section}

\theoremstyle{plain}
\newtheorem{theorem}{Theorem}[section]
\newtheorem{corollary}[theorem]{Corollary}
\newtheorem{lemma}[theorem]{Lemma}

\theoremstyle{definition}
\newtheorem{definition}[theorem]{Definition}
\newtheorem{example}{Example}
\theoremstyle{remark}
\newtheorem*{remark}{Remark}
\newtheorem*{remarks}{Remarks}

\newcommand{\SL}{\text {\rm SL}}
\newcommand{\R}{\mathbb{R}}
\newcommand{\Q}{\mathbb{Q}}
\newcommand{\Z}{\mathbb{Z}}
\newcommand{\N}{\mathbb{N}}
\newcommand{\C}{\mathbb{C}}
\renewcommand{\H}{\mathbb{H}}
\newcommand{\F}{\mathbb{F}}
\newcommand{\leg}[2]{\left( \frac{#1}{#2} \right)}
\newcommand{\calE}{\mathcal{E}}
\newcommand{\calL}{\mathcal{L}}
\newcommand{\calQ}{\mathcal{Q}}
\newcommand{\frake}{\mathfrak e}
\newcommand{\PSL}{\operatorname{PSL}}
\newcommand{\pmd}[1]{\textnormal{ (mod } #1 \textnormal{)}}

\title{Linear congruence relations for exponents of Borcherds products}

\author{Andreas Mono}
\address{Department of Mathematics and Computer Science, Division of Mathematics, University of Cologne, Weyertal 86--90, 50931 Cologne, Germany}
\curraddr{\textsc{Department of Mathematics, Vanderbilt University, 1326 Stevenson Center, Nashville, TN 37240, USA}}
\email{andreas.mono@vanderbilt.edu}
\author{Badri Vishal Pandey}
\address{Department of Mathematics and Computer Science, Division of Mathematics, University of Cologne, Weyertal 86-90, 50931 Cologne, Germany}
\email{badrivishal9451@gmail.com}
\date{\today}

\begin{document}

\begin{abstract} 
For all positive powers of primes $p\geq 5$, we prove the existence of infinitely many linear congruences between the exponents of twisted Borcherds products arising from a suitable scalar-valued weight $1/2$ weakly holomorphic modular form or a suitable vector-valued harmonic Maa{\ss} form. To this end, we work with the logarithmic derivatives of these twisted Borcherds products, and offer various numerical examples of non-trivial linear congruences between them modulo $p=11$. In the case of positive powers of primes $p=2,3$, we obtain similar results by multiplying the logarithmic derivative with a Hilbert class polynomial as well as a power of the modular discriminant function. Both results confirm a speculation by Ono.
\end{abstract}

\subjclass[2020]{11F33 (Primary); 11F03, 11F12, 11F30, 11F37 (Secondary)}

\keywords{Borcherds products, Hilbert class polynomials, Harmonic Maa{\ss} forms, Modular forms}

\thanks{The first author was supported by the CRC/TRR 191 ``Symplectic Structures in Geometry, Algebra and Dynamics'', funded by the DFG (project number 281071066). The second author has received funding from the European Research Council (ERC) under the
European Union’s Horizon 2020 research and innovation programme (grant agreement No. 101001179).}

\maketitle

\section{Introduction and statement of results}

In the 1990s, Borcherds \cite{Borcherds1, Borcherds2} developed a theory of infinite automorphic products known as \emph{Borcherds products} today. His first examples made use of weight 1/2 weakly holomorphic modular forms in the Kohnen plus space for $\Gamma_0(4)$. A weakly holomorphic modular form is a modular form, which is permitted to have poles at the cusps. The Kohnen plus space for $\Gamma_0(4)$, denoted by $M^!_{\frac{1}{2}}(\Gamma_0(4))$, consists of those forms with Fourier expansion
$$
f(\tau) = \sum_{\substack{n\equiv 0, 1\pmd 4\\ n\gg -\infty}} a_f(n)q^n, \qquad q\coloneqq e^{2\pi i \tau}, \qquad \tau \in \H \coloneqq \{\tau=u+iv \in \C \colon v > 0\}.
$$
This space has a nice basis $\{f_0, f_3, f_4,\ldots\}$, where for each non-negative  $d\equiv 0, 3\pmd 4$  there is a unique weight 1/2 weakly holomorphic modular form with Fourier expansion
\begin{equation} \label{fd}
f_d(\tau) = q^{-d} + \sum_{\substack{D \geq 1 \\ D \equiv 0,1 \pmd{4}}} A(D,d) q^D.
\end{equation}

For $f\in M_{\frac{1}{2}}^!(\Gamma_0(4))$ with integer Fourier coefficients and so-called ``Weyl vector'' $\rho_f \in \Q$ associated to $f$, Borcherds proved (see Theorem 14.1 of \cite{Borcherds1}) that the infinite product 
\begin{equation*} 
F_f(\tau) \coloneqq q^{\rho_f} \prod_{n \geq 1} \left(1-q^n\right)^{a_f\left(n^2\right)}
\end{equation*}
is a meromorphic weight $a_f(0)$ modular form on $\SL_2(\Z)$. 
For $d=0$, we have the theta function
$$
12f_0=12+24q+24q^4+24q^9+\ldots,
$$ 
and Borcherds' theorem reproduces the famous infinite product for the weight 12 cusp form
$$
\Delta(\tau)\coloneqq F_{12f_0}(\tau)=q\prod_{n=1}^{\infty}(1-q^n)^{24}.
$$

These infinite products are more surprising for other $d$, as they arise from the Hilbert class polynomials for $\mathcal{Q}_{-d}/\SL_2(\Z)$, the set of $\SL_2(\Z)$-equivalence classes of discriminant $-d$  positive definite integral binary quadratic forms. The number of equivalence classes is called \emph{class number of $\calQ_{-d}$} and is denote by $h(-d)$.
For fundamental discriminants $-d$, 
Borcherds' theorem asserts that
\begin{equation} \label{ClassPoly}
H_{-d}(j(\tau))=F_{f_d}(\tau),
\end{equation}
where $j(\tau)$ is \emph{Klein's $j$-invariant}, namely the function
$$
j(\tau)\coloneqq q^{-1}+744+196884q+21493760q^2+\ldots, 
$$  
and $H_{-d}(X)$ is the \emph{Hilbert class polynomial}
\begin{align*}
H_{-d}(X) \coloneqq \prod_{Q \in \mathcal{Q}_{-d} \slash \SL_2(\Z)} (X-j(\tau_Q))^{1\slash \omega_Q}.
\end{align*}
Here and throughout, the \emph{Heegner point} $\tau_Q$ is the unique point in the complex upper half plane for which $Q(\tau_Q,1)=0$, where $Q(x,y)=ax^2+bxy+cy^2 \in \calQ_{-d}$, and
$\omega_Q\in \{1, 2, 3\}$ is the order of its stabilizer in $\PSL_2(\Z)$. 

\begin{example} \label{exa:f3example}
For $d=3$, we have that
\begin{align*}
f_3(\tau)&=q^{-3}+\sum_{1\leq D\equiv 0, 1 \pmd 4}A(D,3)q^D \\
&= q^{-3}-248q+26752q^4-85995q^5+1707264q^8-4096248 q^9+\ldots.
\end{align*}
After noting that $H_{-3}(X)=X^{\frac{1}{3}}$, we find that
\begin{align*}
H_{-3}(j(\tau))^3 &= j(\tau)=q^{-1}\prod_{D=1}^{\infty}(1-q^D)^{3A(D^2,3)} \\
&=\left(q^{-\frac{1}{3}}(1-q)^{-248}(1-q^2)^{26752}(1-q^3)^{4096248}\cdots\right)^3.
\end{align*}
\end{example}

In subsequent work, Zagier \cite{Zagiertraces}  discovered the phenomenon of \emph{twisted Borcherds products}. To define these, let $-d < 0$ and $D > 1$ be relatively prime fundamental discriminants. Zagier extended Borcherds' characterization of Hilbert class polynomials by allowing twists by the extended genus character $\chi=\chi_{D,-d}$, defined on equivalence classes of integral binary quadratic forms of discriminant $-dD$. Taking values in $\{\pm 1\}$, the character $\chi$ is determined from the primes $p$ represented by $Q\in \mathcal{Q}_{-dD}$ by the rule $\chi(Q)=\leg{D}{p}=\leg{-d}{p}$, see \eqref{eq:genuschardef} as well. Using the modular forms $f_d$ in (\ref{fd}), Zagier obtained (see Theorem 7 of \cite{Zagiertraces}) the modular infinite products
\begin{equation} \label{TwistedClassPoly}
H_{D,-d}(j(\tau))= \prod_{m=1}^{\infty} P_D(q^m)^{A(Dm^2,d)},
\end{equation}
where
\begin{align*}
H_{D,-d}(X) \coloneqq \prod_{Q \in \mathcal{Q}_{-dD} \slash \SL_2(\Z)} \left(X-j(\tau_Q)\right)^{\chi_{dD}(Q)/\omega_Q} \in \Q\left(\sqrt{D}\right)(X),
\end{align*}
and $P_D(t)$ is the rational function
\begin{equation} \label{eq:PDdef}
P_D(t)\coloneqq\prod_{n \pmd{D}}\left (1-\exp\left(\frac{2\pi i n}{D}\right)t\right)^{\leg{D}{n}}=\exp\left(-\sqrt{D}\sum_{r=1}^{\infty} \leg{D}{r} \frac{t^r}{r}\right).
\end{equation}

\begin{remark} Zagier's generalization (\ref{TwistedClassPoly}) gives (\ref{ClassPoly})  when one lets
$D=1$.  
\end{remark}

\begin{example}[Zagier] To illustrate (\ref{TwistedClassPoly}), we consider
 $D=8$ and $d=3$, where
 $Q_1(x,y)=x^2+6y^2$ and $Q_2(x,y)=2x^2+3y^2$  are representatives for the classes of discriminant $-24$ quadratic forms. Their associated Heegner points are $i\sqrt{6}$ and $i\sqrt{3/2}$. 
 Therefore, we have that
$$
H_{8,-3}(j(\tau)) = \frac{j(\tau)-j(i\sqrt{6})}{j(\tau)-j(i\sqrt{3\slash 2})} = \prod_{m=1}^{\infty} \left(\frac{1-\sqrt{2}q^m+q^{2m}}{1+\sqrt{2}q^m+q^{2m}}\right)^{A(8m^2,3)},
$$
as $\chi_{24}(Q_1) = \big(\frac{8}{7}\big) = 1$ and $\chi_{24}(Q_2) = \big(\frac{8}{5}\big) = -1$. 
\end{example}

Bruinier and Ono \cite{BruinierOno} further generalized these infinite products using vector-valued harmonic Maa{\ss} forms (see Definition \ref{defn:vvMaass}). They proved (see Theorem 6.1 of \cite{BruinierOno}) that the space $M^!_{\frac{1}{2}}(\Gamma_0(4))$ of ``input forms'' can be extended to suitable vector-valued weight 1/2 harmonic Maa{\ss} forms $g(\tau)$ using a lattice $L(N)$ corresponding to $\Gamma_0(N)$ and its associated Weil representation $\rho_{L(N)}$ (see Section \ref{sec:BorcherdsProducts} for more details). 
Their extension of (\ref{ClassPoly}) and (\ref{TwistedClassPoly}) gives modular infinite products with (twisted) Heegner divisor on generic $\Gamma_0(N)$ congruence subgroups, and these products are given by the expansion
\begin{equation} \label{PsiD1}
\Psi_{D}(g;\tau)\coloneqq\prod_{m=1}^{\infty} P_{-D}(q^m)^{c_g^+(r\overline{m} ; Dm^2)},
\end{equation}
where $\overline{n}$ denotes the canonical residue class of $n$ modulo $2N$ (see \eqref{eq:lattices}), $r \in Z\Z$ is chosen such that $r^2 \equiv D \pmod{4N}$, and $c_g^+(r\overline{m} ; Dm^2)$ refers to the Fourier coefficients of the components of the holomorphic part of $g$ (see \eqref{eq:Fouriersplit}). The forms $\Psi_{D}$ have a twisted Heegner divisor with discriminant $-D$ (see Definition \ref{def:heegner-divisor}).

The aim of this paper is to study the arithmetic of the exponents of these infinite products, which are designated Fourier coefficients of weight $1/2$ weakly holomorphic modular forms and vector-valued harmonic Maa{\ss} forms. As Hilbert class polynomials are irreducible in $\Z[X]$, and have the property that their splitting fields are the Hilbert class field of the corresponding imaginary quadratic field, at first glance one does not expect any relations between them. Indeed, Borcherds products and twisted Borcherds products with different discriminants are linearly independent. 
However, Ono speculated in private communication that there are many congruence relations. In the case of Ramanujan's third order mock theta functions (see \cite{AMSBook} for example), he confirmed this in \cite{Ono5} with an application to the partition function. Here, we extend his work to the most general setting.
To make this precise, we define the normalized logarithmic derivative of $\Psi_D(g;\tau)$ 
\begin{align} \label{eq:log-derivative-of-borcherds-prod}
\calL_D(g;\tau) \coloneqq \frac{1}{-\sqrt{-D}} \frac{\left(\mathbb{D} \Psi_{D}\right)(g;\tau)}{\Psi_{D}(g;\tau)},
\end{align}
where 
$\mathbb{D} \coloneqq q \frac{\mathrm{d}}{\mathrm{d}q} = \frac{1}{2\pi i} \frac{\mathrm{d}}{\mathrm{d}\tau}$
is a normalized holomorphic differential operator.

Our first result establishes conditions under which these logarithmic derivatives are reductions of holomorphic modular forms.

\begin{theorem}\label{thm:LogDeriv}
Suppose that $g(\tau)$ is a vector-valued weight 1/2 harmonic Maa{\ss} form of level $N$ satisfying the hypotheses of Theorem \ref{thm:BorcherdsProduct} and with $c_g^+(\mu,n) \in \mathcal{O}_K$ for all $n$, $\mu \in L'\slash L$ and some number field $K$. Let $D>1$ be a fundamental discriminant, $j\in\N$, $n_0 < 0$ be the order of $g(\tau)$ at $i\infty$, $p\geq 5$ be a prime for which $\leg{Dm}{p}\in \{0, -1\}$ for all $n_0\leq m\leq -1$, and $\tau_{m,i},\, 1\leq i\leq h(Dm)$ the Heegner points of discriminant $Dm$ for all $n_0\leq m\leq -1$. Then $\calL_D(g;\tau)$ is the reduction modulo $p^j$ of a holomorphic modular form of weight 
\[
2+\left[\SL_2(\Z)\colon \Gamma_0(N)\right] (p-1)p^{j-1}\sum_{m=n_0}^{-1}h(Dm)
\]
on $\Gamma_0(N)$.
\end{theorem}

Theorem \ref{thm:LogDeriv} implies the existence of many nontrivial congruence relations between (twisted) Borcherds exponents for distinct discriminants. To describe them, we let $p$ be a prime and $g$ be a form as in Theorem \ref{thm:LogDeriv}. We define the infinite set of fundamental discriminants
\begin{multline}\label{eq:S_p}
	S_{p,g} 
\coloneqq
\Bigg\{D > 1 \colon D \equiv r^2 \pmd{4N} \text{ for some }r\in \Z,\\ \leg{Dm}{p}\in \{0, -1\} \text{ for all } n_0\leq m\leq -1\Bigg\}
\end{multline}
as well as the number
\begin{align} \label{eq:hSg-def}
h_{S,g} \coloneqq\max\left\{\left[\SL_2(\Z)\colon \Gamma_0(N)\right]\sum_{m=n_0}^{-1}h(Dm)\ \colon  D\in S\right\}
\end{align}
for a finite subset $S \subset S_{p,g}$.
\begin{corollary}\label{cor:CongruenceRelations}
Assume the notation and hypotheses from Theorem \ref{thm:LogDeriv}. Suppose that $K=\Q$, namely $c_g^+(\mu,n) \in \Z$ for all $n$, $\mu \in L'\slash L$. Let $S\subset S_{p,g}$ be a finite subset. Then the following are true:
\begin{enumerate}[leftmargin=*, label={\normalfont(\roman*)}]
\item If $S$ satisfies
	\begin{align*}
	\#S & > \frac{1}{12}\left((p-1)p^{j-1}h_{S,g}+2\right)[\SL_2(\Z)\colon \Gamma_0(N)],
\end{align*}
	then the $q$-series $\left\{\calL_D(g;\tau)\colon D\in S\right\}$ are linearly dependent modulo $p^j$.
\item Let $D_S\coloneqq\max\{D \colon D \in S\}$. If $S$ satisfies
	\begin{align*}
	\hspace*{\leftmargini}	\#S > \frac{[\SL_2(\Z)\colon \Gamma_0(N)]^2}{12\pi}\left(2\pi + (p-1)p^{j-1}\sum_{m=n_0}^{-1}\left(\log(|D_Sm|)+2\right)\sqrt{|D_Sm|}\right) ,
	\end{align*}
 then the $q$-series $\left\{\calL_D(g;\tau)\colon D\in S\right\}$ are linearly dependent modulo $p^j$.
\end{enumerate}
\end{corollary}

\begin{remark}
The existence of such finite subsets $S\subset S_{p,g}$ of the infinite set $S_{p,g}$ follows from the observation that the size of $S$ grows linearly as $D \in S_{p,g}$ increases, while the lower bounds in Corollary \ref{cor:CongruenceRelations} grow sub-linearly as $D \in S_{p,g}$ increases (see Lemma \ref{lem:class-number-bound} for (i)).
\end{remark}

We offer an example of Corollary \ref{cor:CongruenceRelations} (i).

\begin{example} \label{exa:congexample}
We consider the twisted products (\ref{TwistedClassPoly}) arising from $f_3$ (see Example \ref{exa:f3example}) with discriminants contained in 
$$
S \coloneqq \{5, 20, 37, 53, 56, 80, 89, 92, 97, 104\}.
$$
By construction, $f_3 \in \Z((q))$, and the class numbers are 
\begin{align*}
h(-15) &= 2, \quad h(-60) = 2, \quad h(-111) = 8, \quad h(-159) = 10, \quad h(-168) = 4, \\
h(-240) &= 2, \quad h(-267) = 2, \quad h(-276) = 8, \quad h(-291) = 4, \quad h(-312) = 4.
\end{align*}
We consider congruences modulo $p=11$, which is ramified or inert in $\Q(\sqrt{-3D})$ if and only if $D \equiv 0,1,3,4,5,9 \pmd{11}$. Note that each $D \in S$ is not a square, coprime to $3$, a square modulo $4$, and satisfies $D \equiv 1,5,9 \pmd{11}$. By Lemma \ref{lem:logdiff}, the logarithmic derivative of $\Psi_{D}(f_3;\tau)$ is
\begin{align*}
\calL_D(f_3,\tau) = \sum_{r=1}^{\infty} \sum_{n=1}^{\infty} \left(\frac{-D}{r}\right) n A(Dn^2,3) q^{nr}, \qquad D \in S.
\end{align*}
As we have $\#S = 10$ and $((11-1)\cdot 10 + 2) / 12 \in (8,9)$, Corollary \ref{cor:CongruenceRelations} and explicit calculations yield the congruences
\begin{align*}
7\cdot\calL_{5}(f_3,\tau) - 7\cdot\calL_{53}(f_3,\tau) + \calL_{56}(f_3,\tau) + 2\cdot \calL_{89}(f_3,\tau) &\equiv 0 \pmod{11}, \\
3\cdot\calL_{20}(f_3,\tau) + 2\cdot\calL_{37}(f_3,\tau) + 6\cdot \calL_{80}(f_3,\tau) + 2\cdot \calL_{92}(f_3,\tau) + 3\cdot \calL_{97}(f_3,\tau) &\equiv 0 \pmod{11}, \\
3\cdot \calL_{5}(f_3,\tau) + 2\cdot \calL_{89}(f_3,\tau) &\equiv 0 \pmod{11}, \\
3\cdot\calL_{53}(f_3,\tau) - 2\cdot\calL_{56}(f_3,\tau) - 2\cdot\calL_{89}(f_3,\tau) &\equiv 0 \pmod{11}, \\
\calL_{56}(f_3,\tau) + 2\cdot\calL_{89}(f_3,\tau) - 2\cdot\calL_{92}(f_3,\tau) + 2\cdot \calL_{104}(f_3,\tau) &\equiv 0 \pmod{11}, \\
2\cdot\calL_{5}(f_3,\tau) - 2\cdot\calL_{53}(f_3,\tau) - \calL_{92}(f_3,\tau) + \calL_{104}(f_3,\tau) &\equiv 0 \pmod{11}.
\end{align*}
These congruences are non-trivial in the sense that none of the $\calL_D(f_3,\tau)$ for $D \in S$ satisfies $\calL_D(f_3,\tau) \equiv 0 \pmd{11}$ itself. We provide our numerical computations related to these examples in an appendix at the end of the paper.
\end{example}

For $p\in\{2,3\}$ we modify $\calL_{D}(\tau)$ to get an analogous result to Corollary \ref{cor:CongruenceRelations}. To this end, we adapt an approach of Ono from \cite{Ono4, Ono5}, which he used to study the partition function via mock theta functions and Borcherds products. For each $D\in S$, $p\in\{2,3\}$ and $j\in\Z^{+}$, we define
\begin{align}\label{eq:L-cap-2-3}
	\widehat{\calL}_{D,p^j}(g;\tau)\coloneqq\calL_D(g;\tau) \Delta(\tau)^{h_{S,g}  p^{j-1}} \prod_{m=n_0}^{-1}H_{-Dm}(1/\Delta(\tau))^{\left[\SL_2(\Z)\colon\Gamma_0(N)\right] p^{j-1}},
\end{align}
and recall the set $S_{p,g}$ as well as the number $h_{S,g}$ from equations \eqref{eq:S_p} and \eqref{eq:hSg-def}.

\begin{theorem}\label{thm:LogDeriv2and3}
Let $p \in \{2,3\}$. Let $g(\tau)$ be a vector-valued weight 1/2 harmonic Maa{\ss} form of level $N$ satisfying the hypotheses of Theorem \ref{thm:BorcherdsProduct}, with $c_g^+(\mu,n) \in \mathcal{O}_K$ for all $n$, $\mu \in L'\slash L$ and some number field $K$. Suppose that $g$ has order $n_0 < 0$ at $i\infty$. Let $S \subset S_{p,g}$ be a finite subset such that
$$ 
\# S \geq \frac{(12 h_{S,g} p^{j-1}+2) \left[\SL_2(\Z)\colon\Gamma_0(N)\right]}{12}.
$$
Then, the forms in $\{\widehat{\calL}_{D,p^j}(g;\tau)\colon D\in S\}$ are linearly dependent modulo $p^{j+1}$ for $p=2$ and modulo $p^j$ for $p=3$.
\end{theorem}

\begin{remark}
Choosing $N=6$, $p=2$ and $j=1$ in Theorem \ref{thm:LogDeriv2and3} yields Ono's Theorem 1.2 \cite{Ono5}. 
\end{remark}

This paper is organized as follows. In Section \ref{sec:BorcherdsProducts}, we summarize some background on vector-valued harmonic Maa{\ss} forms, Heegner divisors, and generalized twisted Borcherds products following Bruinier and Ono \cite{BruinierOno}. In Section \ref{sec:proofs}, we recall some facts on logarithmic derivatives of such Borcherds products, on the classical modular Eisenstein series, and on class numbers of imaginary quadratic fields. This established, we prove Theorems \ref{thm:LogDeriv}, \ref{thm:LogDeriv2and3} and Corollary \ref{cor:CongruenceRelations}. Finally, we collect our computations related to Example \ref{exa:congexample} in an appendix at the end of the paper.

\section*{Acknowledgements} 
\noindent Both authors thank Ken Ono for helpful discussions and comments, and Markus Schwagenscheidt as well as the anonymous referee for valuable comments on an earlier version of this paper.

\section{Generalized Borcherds products}\label{sec:BorcherdsProducts}

Here we recall the basic framework of  Borcherds products that arise from modular forms and harmonic Maa{\ss} forms. The first examples of such automorphic infinite products were introduced by Borcherds \cite{Borcherds1, Borcherds2}, and their extension to the setting of harmonic Maa{\ss} forms was obtained by Bruinier and Ono \cite{BruinierOno}. For more background on (vector-valued) harmonic Maa{\ss} forms, we refer the reader to the seminal work by Bruinier and Funke \cite{BruinierFunke} and the exposition \cite{AMSBook}.

\subsection{Vector-valued harmonic Maa{\ss} forms}
The metaplectic double cover of $\Gamma \coloneqq \mathrm{SL}_2(\Z)$ is given by
\begin{align*}
\widetilde{\Gamma} \coloneqq \text{Mp}_2(\Z) \coloneqq \left\{ (\gamma, \phi) \colon \gamma = \left(\begin{smallmatrix} a & b \\ c & d \end{smallmatrix}\right)\in \mathrm{SL}_2(\Z), \ \phi\colon \H \rightarrow \C \text{ holomorphic}, \ \phi^2(\tau) = c\tau+d  \right\},
\end{align*}
and $\widetilde{\Gamma}$ is generated by the pairs
\begin{align*}
\widetilde{T} \coloneqq \left(\left( \begin{matrix} 1 & 1 \\ 0 & 1 \end{matrix} \right),1\right), \qquad \widetilde{S} \coloneqq \left(\left( \begin{matrix} 0 & -1 \\ 1 & 0 \end{matrix}\right) ,\sqrt{\tau}\right),
\end{align*}
where we fix the principal branch of the complex square root throughout.
To introduce the Weil representation, let $L$ be an even lattice of signature $(r,s)$, and $\mathfrak{Q}$ be a quadratic form on $L$ with associated bilinear form $(\cdot,\cdot)_{\mathfrak{Q}}$. The lattice $L$ has a dual lattice $L'$, which gives rise to the group ring $\C[L'\slash L]$ with standard basis $\frake_{\mu}$ for $\mu \in L'\slash L$. 
\begin{definition}
The elements $\widetilde{T}$ and $\widetilde{S}$ act through the Weil representation $\rho_L$ associated to $L$ as
\begin{align} \label{eq:Weil}
\rho_L\left(\widetilde{T}\right)(\frake_\mu) \coloneqq e^{2\pi i \mathfrak{Q}(\mu)} \frake_\mu, \qquad
\rho_L\left(\widetilde{S}\right)(\frake_\mu) \coloneqq \frac{e^{2\pi i \frac18(s-r)}}{\sqrt{|L'\slash L|}} \sum_{\nu \in L'\slash L} e^{-2\pi i (\nu,\mu)_{\mathfrak{Q}}} \frake_{\nu},
\end{align}
and on general $\widetilde{\gamma} \in \widetilde{\Gamma}$ by multiplicative extension. The dual Weil representation $\overline{\rho_L}$ is given by complex conjugation of the right hand sides in \eqref{eq:Weil}.
\end{definition}

For positive integers $N$ we consider the lattices
\begin{align} \label{eq:lattices}
L &\coloneqq L(N) \coloneqq \left\{\left(\begin{matrix} b & -\frac{a}{N} \\ c & -b \end{matrix}\right) \colon a,b,c \in \Z\right\}, \nonumber\\
L' &\coloneqq L(N)' \coloneqq \left\{\left(\begin{matrix} \frac{b}{2N} & -\frac{a}{N} \\ c & -\frac{b}{2N} \end{matrix}\right) \colon a,b,c \in \Z\right\}
\end{align}
corresponding to $\Gamma_0(N)$. The signature of $L(N)$ is $(2,1)$ and it holds that $L(N)' \slash L(N) \cong \Z\slash2N\Z$. We equip $L(N)' \slash L(N)$ with the quadratic form $Q_N(x) = -x^2 \slash (4N)$. The associated bilinear form is $(x,y)_{Q_N} = -(xy)\slash(4N)$.

We now move to the definition of vector-valued harmonic Maa{\ss} forms. Let $f:\H\to \C[L'/L]$ be a smooth function and $(\gamma,\phi) \in \widetilde{\Gamma}$. We first define the \emph{slash-operator} as
\begin{align*}
f \vert_{\kappa}(\gamma,\phi) \coloneqq \phi(\tau)^{-2\kappa}\rho_{L}^{-1}(\gamma,\phi)f(\gamma\tau).
\end{align*}
Then, a vector-valued harmonic Maa{\ss} form is defined as follows.
\begin{definition} \label{defn:vvMaass}
Let $g\colon \H \rightarrow \C[L' \slash L]$ be smooth. Then $g$ is a vector-valued harmonic Maa{\ss} form of weight $\kappa \in \frac{1}{2}\Z$ with respect to $\rho_L$ if it satisfies the following three conditions.
\begin{enumerate}[leftmargin=*, label=(\roman*)]
\item For every $\tau \in \H$ and every $(\gamma,\phi) \in \widetilde{\Gamma}$, we have
$g\vert_{\kappa}(\gamma,\phi)(\tau) = g(\tau)$.
\item The function $g$ is harmonic with respect to the weight $\kappa$ hyperbolic Laplace operator, namely ($\tau = u+iv$)
\begin{align*}
0 = (\Delta_{\kappa}g)(\tau) \coloneqq \left(\left(-v^2\left(\frac{\partial^2}{\partial u^2} + \frac{\partial^2}{\partial v^2} \right) + i\kappa v\left(\frac{\partial}{\partial u} + i\frac{\partial}{\partial v} \right)\right)g\right)(\tau).
\end{align*}
\item There exists a polynomial $P_g \in \C[L' \slash L]\big[q^{-1/(4N)}\big]$ such that $g(\tau) - P_g(q) \in O\big(e^{-\delta v}\big)$ as $v \to \infty$ for some $\delta > 0$.
\end{enumerate}
\end{definition}

Bruinier and Funke \cite{BruinierFunke} proved that the Fourier expansion of a weight $\kappa \neq 1$ vector-valued harmonic Maa{\ss} form $g$ for $\rho_L$ splits as 
\begin{align} \label{eq:Fouriersplit}
\begin{split}
g(\tau) &= \sum_{\mu \in L' \slash L} g_{\mu}^+(\tau) \frake_{\mu} + \sum_{\mu \in L' \slash L} g_{\mu}^-(\tau) \frake_{\mu}, \\
g_{\mu}^+(\tau) &\coloneqq \sum_{\substack{n \in \Z \\ n \gg - \infty}} c_g^+(\mu,n) q^{\frac{n}{4N}}, \qquad g_{\mu}^-(\tau) \coloneqq \sum_{\substack{n \in \Z \\ n < 0}} c_g^-(\mu,n) \Gamma\left(1-\kappa,\frac{\pi|n|v}{N}\right)q^{\frac{n}{4N}},
\end{split}
\end{align}
where $\Gamma(t,x) \coloneqq \int_x^\infty u^{t-1} e^{-u} du$ refers to the \emph{incomplete Gamma function}. 

The functions 
\begin{align*}
\sum_{\mu \in L' \slash L} g_{\mu}^+(\tau) \frake_{\mu}, \qquad \sum_{\mu \in L' \slash L} g_{\mu}^-(\tau) \frake_{\mu}
\end{align*}
are called the \emph{holomorphic part} and \emph{nonholomorphic part} of $g$ respectively. By definition, every weakly holomorphic modular form is a harmonic Maa{\ss} form with trivial nonholomorphic part. From the transformation behaviour under $\widetilde{T}$, we deduce that $c_g^{\pm}(\mu,n) = 0$ unless $n \equiv \mu^2 \pmd{4N}$.

\begin{remark}
We note that a scalar-valued modular form $f$ of weight $\kappa \in \Z$ and level $N$ gives rise to a vector-valued modular form
\begin{align*}
\sum_{\gamma \in \Gamma_0(N) \backslash \Gamma} \left(f \vert_{\kappa} \gamma\right)(\tau) \rho_{L}\left(\gamma^{-1}\right) \frake_{0}
\end{align*}
of weight $\kappa$ for $\rho_{L}$, which is studied in detail by Scheithauer \cite[Theorem 5.4]{scheit}. Lifting $\eta$-quotients was done by Borcherds in \cite[Theorem 6.2]{Borcherds3}.
\end{remark}

\subsection{Heegner divisors}
Let $\calQ_{N,-dD}$ be the subset of $\calQ_{-dD}$ consisting of forms $Q(x,y) = ax^2+bxy+y^2$ satisfying $N \mid a$. Then, the congruence subgroup $\Gamma_0(N)$ acts on $\calQ_{N,-dD}$ by
\begin{align*}
\left(Q \circ \left(\begin{matrix} a & b \\ c & d \end{matrix}\right)\right)(x,y) \coloneqq Q(ax+by, cx+dy), \qquad \left(\begin{matrix} a & b \\ c & d \end{matrix}\right) \in \Gamma_0(N).
\end{align*}
This action preserves the discriminant $-dD$, and has finitely many orbits whenever $-dD \neq 0$.
Let $-d<0$ and $D>1$ be coprime fundamental discriminants such that both $D$ and $-d$ are squares modulo $4N$. Following Gross, Kohnen, and Zagier \cite[p. 508]{GKZ}, we define the \emph{extended level $N$ Genus character} by
\begin{align} \label{eq:genuschardef}
\chi_{-dD}\left([a,b,c]\right) &\coloneqq \begin{cases}
\left(\frac{D}{n}\right) & \text{if } \gcd{(a,b,c,d)} = 1, [a,b,c] \text{ represents } n, \gcd{(d,n)} = 1, \\
0 & \text{if } \gcd{(a,b,c,d)} > 1,
\end{cases}
\end{align}
where $\leg{D}{\cdot}$ is the Kronecker character. Given $\mu \in L'\slash L$, let 
\begin{align*}
\calQ_{N,D,\mu} & \coloneqq \left\{Q=[a,b,c] \in \calQ_{N,D} \ \colon \ b \equiv \mu \pmd{2N} \right\}.
\end{align*}
Now, we are in position to define Heegner Divisors. To this end, we follow the paper \cite{BruinierOno}.
\begin{definition}[Heegner Divisors]\label{def:heegner-divisor} Assume the notation and hypotheses above.
\begin{enumerate}[leftmargin=*]
\item Let $\mu \in L'\slash L$, and  $m \in \Z$. The \emph{twisted Heegner divisor of level $N$ and discriminant $D$} is given by
\begin{align*}
Z_{N,D}(m,\mu) \coloneqq \sum_{Q \in \calQ_{N,Dm,\mu} \slash \Gamma_0(N)} \chi_{Dm}(Q) \frac{\tau_Q}{\omega_Q},
\end{align*} 
where $\tau_Q$ is the unique Heegner point corresponding to $Q$ in $\H$.
\item If $g$ is a vector-valued harmonic Maa{\ss} form for $\rho_L$ then the \emph{twisted Heegner divisor associated to $g$} is given by
\begin{align*}
Z_{N,D}(g) \coloneqq \sum_{\mu \in L' \slash L} \sum_{m < 0} c_g^+(\mu,m) Z_{N,D}(m,\mu).
\end{align*}
\end{enumerate}
\end{definition}

\subsection{Generalized twisted Borcherds products}
We let $D > 1$ be a fundamental discriminant, and recall the auxiliary function $P_D(X)$ from \eqref{eq:PDdef}.

Let $\overline{n}$ denote the canonical residue class of $n$ modulo $2N$. The following theorem of Bruinier and Ono (see Theorem 6.1 of \cite{BruinierOno}), offers the harmonic Maa{\ss} form extension of the automorphic infinite products first assembled by Borcherds (alluded to in the introduction) to which we refer as \emph{generalized twisted Borcherds products}.

\begin{theorem}\label{thm:BorcherdsProduct}
Choose $r\in \Z$ such that $D \equiv r^2 \pmd{4N}$. Let $g(\tau)$ be a weight $1/2$ vector-valued harmonic Maa{\ss} form for $\rho_L$ with Fourier coefficients as in \eqref{eq:Fouriersplit}. Suppose that $c_g^+(\mu,n) \in \R$ for all $n$, $\mu \in L'/L$, and $c_g^+(\mu,n) \in \Z$ for all $n \leq 0$. Then, $\Psi_{D}(g;\tau)$ as defined in \eqref{PsiD1} converges whenever $\mathrm{Im}(\tau)$ is sufficiently large, and defines a meromorphic modular form of weight $0$ for $\Gamma_0(N)$ with twisted Heegner divisor given by $Z_{N,D}(g)$. 
\end{theorem}

Since $D\neq1$, we deduce from \cite[Section $4$]{BruinierOno} that the divisor supported at the cusps of $\Gamma_0(N)$ vanishes.
\begin{remarks}
\
\begin{enumerate}[leftmargin=*]
\item The Borcherds product $\Psi_{D}(g;\tau)$ transforms with a multiplier system of infinite order in general.
\item An example of a generalized twisted Borcherds product arising of a vector-valued harmonic Maa{\ss} form is given in Section 8.2 of \cite{BruinierOno}. Its twisted Heegner divisor is provided after Lemma 8.1 there.
\end{enumerate}
\end{remarks}

\section{Proof of Theorems~\ref{thm:LogDeriv}, \ref{thm:LogDeriv2and3} and Corollary~\ref{cor:CongruenceRelations}} \label{sec:proofs}

\subsection{Logarithmic derivatives of generalized Borcherds products}
We have the following properties of $\calL_D(g;\tau)$ defined in \eqref{eq:log-derivative-of-borcherds-prod}, which can be found in Section 2.3 of \cite{Ono4} as well.
\begin{lemma}[\protect{\cite[Lemma 3.1]{Ono5}}] \label{lem:logdiff}
Assume the notation and hypotheses of Theorem \ref{thm:BorcherdsProduct}.
\begin{enumerate}[leftmargin=*, label=\normalfont(\roman*)]
\item The logarithmic derivative $\calL_D(g;\tau)$ of $\Psi_{D}(g;\tau)$ is a meromophic modular form of weight $2$ and level $N$. Its $q$-expansion is given by
\begin{align*}
\calL_D(g;\tau) = \sum_{m=1}^{\infty} c_g^+(r\overline{m};Dm^2) m \sum_{\substack{n=1 \\ \gcd(n,D) = 1}}^{\infty} \left(\frac{-D}{n}\right) q^{mn}.
\end{align*}
\item The poles of $\calL_D(g;\tau)$ are all simple and are located among the $\Gamma_0(N)$ Heegner points corresponding to quadratic forms in the finite union
\begin{align*}
\bigcup_{m = n_0}^{-1} \calQ_{N,Dm} = \bigcup_{\mu \in L' \slash L} \bigcup_{m = n_0}^{-1} \calQ_{N,Dm,\mu},
\end{align*}
where $n_0 < 0$ is the order of $g$ at $i\infty$.
\end{enumerate}
\end{lemma}

\subsection{Eisenstein series of integer weight}
Here we collect all the facts regarding the classical Eisenstein series necessary for our proofs. As usual, for $k\geq 4$ an even integer, let $E_k(\tau)$ denote the Eisenstein series
\begin{equation}\label{eq:eisenstein}
	E_k(\tau)\coloneqq 1 - \frac{2k}{B_k}\sum_{n=1}^{\infty} \sigma_{k-1}(n)q^n,
\end{equation}
where $B_k$ is the $k$-th Bernoulli number and $\sigma_{k-1}$ is the $(k-1)$-th power divisor sum. We recall that $E_k(\tau)$ is a holomorphic modular form of weight $k$ and level $1$. Then we have the following lemma.
\begin{lemma}[\protect{\cite[Lemma 2.3]{BruinierOno2}}]\label{lem:PropEisensteinSeries}
	Suppose that $k\geq 4$ even. 
\begin{enumerate}[leftmargin=*, label=\normalfont(\roman*)]
	\item We have $E_k(\tau)\equiv 1\pmd{24}$.
	\item If $p\geq 5$ is prime and $(p-1) \mid k$, then $E_k(\tau)\equiv 1\pmd{p}$.
	\item If $k\not\equiv 0\pmd{3}$, then $E_k\left(\frac{1+\sqrt{-3}}{2}\right) = 0$.
	\item If $k\equiv 2\pmd{4}$, then $E_k(i)=0$.
	\end{enumerate}
\end{lemma}
The proof relies on the von Staudt--Clausen theorem on the divisibility of denominators of Bernoulli numbers \cite{IrelandRosen}.

\subsection{Class numbers of imaginary quadratic fields}
We require an explicit upper bound on $h(-D)$, which we provide in the following lemma.
\begin{lemma}[\protect{\cite[Lemma 2.2]{gronts}}]\label{lem:class-number-bound}
Let $-d<0$ be a fundamental discriminant. Then, we have
	\[
		h(-d)\leq\frac{\sqrt{d}(\log d+2)}{\pi}.
	\]
\end{lemma}


\subsection{Proof of Theorem~\ref{thm:LogDeriv} and Corollary~\ref{cor:CongruenceRelations}}
In the first proof, we construct the required holomorphic modular form using the knowledge of modular forms mod $p$. Here we adapt a proof by Bruinier and Ono \cite[Theorem 2]{BruinierOno2}. The second result follows from a well-known theorem of Sturm \cite{Sturm} on congruences between holomorphic modular forms.

\begin{proof}[Proof of Theorem~\ref{thm:LogDeriv}]
We assume the notations from the previous sections and that $g$ is a vector valued modular form on $\C[L'/L]$ as in Theorem \ref{thm:BorcherdsProduct} with $c_g^+(\mu,n) \in \mathcal{O}_K$ for all $n$, $\mu \in L'\slash L$ and some number field $K$. Let $D > 1$ be a fundamental discriminant. Let $r \in \Z$ be such that $D\equiv r^2\pmd{4N}$. By Theorem \ref{thm:BorcherdsProduct} and Lemma \ref{lem:logdiff}, we have that $\calL_D(g;\tau)$ is a weight 2 meromorphic modular form on $\Gamma_0(N)$ with all simple poles at Heegner points corresponding to quadratic forms $\calQ_{N,Dm}, n_0\leq m\leq -1$. Let $p\geq 5$ be an inert or ramified prime in $Q(\sqrt{Dm})$ for all $n_0\leq m\leq -1$. 

\noindent\textbf{Case 1:} Assume that 
\[
\prod_{m=n_0}^{-1}\prod_{i=1}^{h(Dm)} j(\tau_{m,i})\left(j(\tau_{m,i})-1728\right)\not\equiv 0\pmod{p}.
\]
This ensures that the reduction modulo (a suitable place above) $p$ of $j(\tau_{m,i})$ defines a supersingular $j$-invariant in $\overline{\F}_p$ (see \cite[p.~8]{BruinierOno2}). Given the Heegner points corresponding to discriminants $Dm$ in $\Gamma_0(N)\backslash \H$, we construct a $p$-integral modular form $\calE_{Dm}(\tau)$ of weight $h(Dm)(p-1)\left[\SL_2(\Z)\colon \Gamma_0(N)\right]$ on $\SL_2(\Z)$. Given a Heegner point in $\SL_2(\Z)\backslash \H$, there are $[\SL_2(\Z)\colon\Gamma_0(N)]$ many $\SL_2(\Z)$-equivalent points in $\Gamma_0(N)\backslash \H$. So given a discriminant $Dm$, there are 
\begin{align*}
a_{Dm,N}\coloneqq [\SL_2(\Z)\colon\Gamma_0(N)] h(Dm)
\end{align*}
Heegner points in $\Gamma_0(N)\backslash \H$, say $\tau_1,\tau_2,\ldots,\tau_{a_{Dm,N}}$. Given these Heegner points, a famous observation of Deligne (see, for example \cite[p.~8]{BruinierOno2}, \cite{serre1}, \cite[Th.~1]{kanekozagier}) implies that the $j$-invariant in characteristic $p$ is the reduction of $j(Q)$ modulo $p$ for some point $Q$ which is a zero of the Eisenstein series $E_{p-1}(\tau)$ of weight $p-1$. Therefore, there are points $Q_1,Q_2,\ldots,Q_{a_{Dm,N}}$ in $\Gamma_0(N)\backslash \H$ for which $E_{p-1}(Q_i)=0$, for all $1\leq i\leq a_{Dm,N}$, with the additional property that
\[
	\prod_{i=1}^{a_{Dm,N}} \left(j(\tau)-j(Q_i)\right)\equiv \prod_{i=1}^{a_{Dm,N}} \left(j(\tau)-j(\tau_i)\right)\pmod{p}
\]
in $\F_p[X]$. Now we define $\calE_{Dm}$ by
\[
	\calE_{Dm}(\tau)\coloneqq \prod_{i=1}^{a_{Dm,N}} \left(E_{p-1}(\tau) \frac{\left(j(\tau)-j(\tau_i)\right)}{\left(j(\tau)-j(Q_i)\right)}\right).
\]
Then we have that $\calE_{Dm}(\tau)$ is a holomorphic modular form of weight 
$$
[\SL_2(\Z)\colon\Gamma_0(N)] h(Dm)(p-1)
$$ 
on $\SL_2(\Z)$ with the additional properties that 
$$
\calE_{Dm}(\tau)\equiv 1\pmd{p}, \qquad \calE_{Dm}(\tau_i)=0
$$
for all $1\leq i\leq a_{Dm,N}$. And hence we have the congruence
\[
	\calE_{Dm}(\tau)^{p^{j-1}}\equiv 1\pmod{p^j}
\]
for all $j\in\Z^+$ (see (3.1) of \cite{BruinierOno2}). We obtain that
\begin{align} \label{eq:thm1.1explicitcongruence}
\calL_D(g;\tau)\equiv \calL_D(g;\tau)\prod_{m=n_0}^{-1}\calE_{Dm}(\tau)^{p^{j-1}}\pmod{p^j}
\end{align}
and the right hand side is a holomorphic modular form of weight 
$$
2+\left[\SL_2(\Z)\colon \Gamma_0(N)\right] (p-1)p^{j-1}\sum_{m=n_0}^{-1}h(Dm)
$$ 
on $\Gamma_0(N)$.

\noindent\textbf{Case 2:} Assume that 
\[
	\prod_{m=n_0}^{-1}\prod_{i=1}^{h(Dm)} j(\tau_{m,i})\left(j(\tau_{m,i})-1728\right)\equiv 0\pmod{p}.
\]
Now, fix an $n_0\leq m\leq -1$. 

\noindent\textbf{Case 2.1:} First assume that
\[
\prod_{i=1}^{h(Dm)} j(\tau_{m,i})=0 \,\,\, \left(\text{ resp. }\prod_{i=1}^{h(Dm)}\left(j(\tau_{m,i})-1728\right)= 0\right),
\]
which implies that $h(Dm)=1$ and $\tau_{m,1}=\frac{1+\sqrt{-3}}{2}$ (resp. $\tau_{m,1}=i$). Using the fact that $\left(\frac{Dm}{p}\right)=-1$ and Lemma \ref{lem:PropEisensteinSeries}, we see that the function $\calE_{Dm}(\tau)\coloneqq E_{p-1}(\tau)$ satisfies the same properties as in the end of Case $1$ and therefore the same conclusion holds for $\mathcal{L}_D(g;\tau)$.

\noindent\textbf{Case 2.2:} Now we assume that
	\[
		\prod_{i=1}^{h(Dm)} j(\tau_{m,i})\left(j(\tau_{m,i})-1728\right)\neq 0
	\]
	whenever $\Q(\sqrt{Dm})\notin \{\Q(i),\Q(\sqrt{-3})\}$. Then since $\{j(\tau_{m,1}),j(\tau_{m,2}),\ldots j(\tau_{m,h(Dm)})\}$ form a Galois orbit over $\Q$, Deuring's result (see \cite[Th.~13.21]{cox}, \cite{deuring}) implies that if
	\begin{align*}
	\prod_{i=1}^{h(Dm)} j(\tau_{m,i})\equiv 0\pmod{p}
	\end{align*}
	then $p\equiv 2\pmod{3}$, and if
	\begin{align*}
	\prod_{i=1}^{h(Dm)} \left(j(\tau_{m,i})-1728\right)\equiv 0\pmod{p}
	\end{align*}
	then $p\equiv 3\pmod{4}$. The same conclusions hold when $\Q\left(\sqrt{Dm}\right)=\Q\left(\sqrt{-3}\right)$ (resp. $\Q\left(\sqrt{Dm}\right)=\Q(i)$) provided that $p\nmid Dm$. In these cases, a simple modification to the proof of the first case gives us the desired result (using Lemma \ref{lem:PropEisensteinSeries}).
\end{proof}

Now, we move to the proof of Corollary~\ref{cor:CongruenceRelations}.
\begin{proof}[Proof of Corollary~\ref{cor:CongruenceRelations}]
Assume that $g$ is as in Theorem \ref{thm:LogDeriv}. Then, we have
\begin{align*}
\calL_D(g;\tau) \equiv G_D \pmod{p^j},
\end{align*}
where $G_D$ is a holomorphic modular form of weight 
$$2+(p-1)p^{j-1}\sum_{m=n_0}^{-1}h(Dm)$$ 
on $\Gamma_0(N)$.

\begin{enumerate}[leftmargin=*, label=(\roman*)]
\item A theorem of Sturm (see Theorem 2.58 of \cite{ono:cbms}) asserts that two general modular forms coincide if their Fourier coefficients agree up to a certain index. This upper bound on the index is called Sturm's bound. In general, Sturm's bound is of the shape $\frac{\kappa}{12} \cdot \left[\SL_2(\Z)\colon \Gamma_0(N)\right]$, where $\kappa$ denotes the weight of the modular forms in question. In our case, Sturm's bound is given by
\begin{align*}
	\frac{1}{12}\left(\left[\SL_2(\Z)\colon \Gamma_0(N)\right] (p-1)p^{j-1}\sum_{m=n_0}^{-1}h(Dm)+2\right)[\SL_2(\Z)\colon \Gamma_0(N)].
\end{align*}
Recall $S_{p,g}$ from \eqref{eq:S_p} and let $S\subset S_{p,g}$ be a finite subset. Further, recall from \eqref{eq:hSg-def} that
\begin{align*}
h_{S,g} = \max\left\{\left[\SL_2(\Z)\colon \Gamma_0(N)\right]\sum_{m=n_0}^{-1}h(Dm)\colon D\in S\right\}.
\end{align*}
Suppose now that $S$ satisfies
\begin{align}\label{eq:S}
	\#S > \frac{(p-1)p^{j-1}h_{S,g}+2}{12}[\SL_2(\Z)\colon \Gamma_0(N)].
\end{align}
Recall from \eqref{eq:thm1.1explicitcongruence} that
\begin{align*}
\calL_D(g;\tau)\prod_{m=1}^{n_0}\calE_{Dm}(\tau)^{p^{j-1}}
\end{align*}
has no poles by construction and hence is a cusp form by Lemma \ref{lem:logdiff}. As $D \in S$ varies, the biggest resulting weight of those cusp forms is $(p-1)p^{j-1}h_{S,g}+2$ by definition of $h_{S,g}$. In turn, this weight yields the largest Sturm bound 
\begin{align} \label{eq:sturmbound}
\frac{(p-1)p^{j-1}h_{S,g}+2}{12}[\SL_2(\Z)\colon \Gamma_0(N)]
\end{align}
among the cusp forms in the finite set
\begin{align*}
\left\{\calL_D(g;\tau)\prod_{m=1}^{n_0}\calE_{Dm}(\tau)^{p^{j-1}}\colon D\in S\right\}.
\end{align*}
By assumption \eqref{eq:S}, that finite set contains strictly more cusp forms than the Sturm bound \eqref{eq:sturmbound}. Hence there is a linear dependence among $\left\{\calL_D(g;\tau)\colon D\in S\right\}$ modulo $p^j$.

\item Recall $D_S = \max\{D\in S\}$. Suppose that
\begin{align*}
\hspace*{\leftmargini}	\#S > \frac{[\SL_2(\Z)\colon \Gamma_0(N)]^2}{12\pi}\left(2\pi+(p-1)p^{j-1}\sum_{m=n_0}^{-1}\left(\log(|D_Sm|)+2\right)\sqrt{|D_Sm|}\right).
\end{align*}
Then, we obtain
\begin{align*}
	h_{S,g} & \leq \frac{\left[\SL_2(\Z)\colon \Gamma_0(N)\right]}{\pi} \sum_{m=n_0}^{-1} \sqrt{|D_Sm|}(\log(|D_Sm|)+2)
\end{align*}
by Lemma \ref{lem:class-number-bound} and therefore \eqref{eq:S} holds. This established, the claim follows as in part (i). \qedhere
\end{enumerate}
\end{proof}

\subsection{Proof of Theorem \ref{thm:LogDeriv2and3}}
Now, we prove Theorem \ref{thm:LogDeriv2and3}. We utilize an idea by Ono \cite[Lemma 3.4]{Ono4}. We again construct an integer weight holomorphic modular form which is congruent to the normalization of the logarithmic derivative of the Borcherds product modulo $p$ and then use Sturm's bound.
\begin{proof}[Proof of Theorem \ref{thm:LogDeriv2and3}]
The function $\calL_D(g;\tau)$ is a level $N$ and weight 2 meromorphic modular form with simple poles at Heegner points of discriminant $Dm$ in $\Gamma_0(N)\backslash\H$ where $n_0\leq m\leq -1$. Hence, we have that
	\[
		\calL_D(g;\tau) \Delta(\tau)^{h_{S,g}} \left(\prod_{m=n_0}^{-1}H_{Dm}(j(\tau))\right)^{[\SL_2(\Z):\Gamma_0(N)]}
	\]
	is a weight $12 h_{S,g}+2$ holomorphic modular form on $\Gamma_0(N)$. Moreover, Lemma \ref{lem:PropEisensteinSeries} yields
	\[
		j(\tau)=\frac{E_4(\tau)^3}{\Delta(\tau)}\equiv \frac{1}{\Delta(\tau)} \pmod{r}
	\]
	for $r\in\{3,8\}$. Let $p\in\{2,3\}$. Since $a\equiv 1\pmd{b}$ implies $a^{b^{j-1}}\equiv 1\pmd{b^{j}}$, we deduce that $\widehat{\calL}_{D,p^j}(\tau)$ defined in \eqref{eq:L-cap-2-3} is a weight $12 h_S p^{j-1}+2$ holomorphic modular form modulo $p^{j+1}$ for $p=2$ and $p^{j}$ for $p=3$ on $\Gamma_0(N)$. We conclude by virtue of Sturm's bound as in the proof of Corollary \ref{cor:CongruenceRelations}.
\end{proof}

\appendix

\section*{Appendix}
Here, we offer some numerical data related to Example \ref{exa:congexample}. We compute $f_3(\tau)$ (see Example \ref{exa:f3example}) up to $q^{10001}$ by the following code implemented in Mathematica \cite{wolfram}.
\begin{lstlisting}[language=Mathematica, caption=Coefficients of f3, captionpos=b]
In[1]:= N1 := 10001
In[2]:= N4 := 2501
In[3]:= sqrtN := 101
In[4]:= Eisen10[q_, l_] := 1 - 264*Sum[DivisorSigma[9, n]*q^n, 
					{n,1,l}]
In[5]:= Eisen104[q_, l_] := Eisen10[q^4,l]
In[6]:= theta[q_, l_] := 1 + 2*Sum[q^(n^2), {n,1,l}]
In[7]:= thetaprime[q_, l_] := 2*Sum[n^2*q^(n^2), {n,1,l}]
In[8]:= Eisen10prime[q_, l_] := Rule[D[Eisen10[q, l], q] /. q,  
					q^4]
In[9]:= Delta4[q_, l_] := q^4*Product[(1-q^(4*n))^(24), {n,1,l}]
In[10]:= Series[((-1)/10) 
		* (Series[(theta[q,sqrtN]*q^4*Eisen10prime[q,N1] 
			- 5*thetaprime[q,sqrtN]*Eisen104[q,N4]) 
				/ Delta4[q,N4], {q,0,N1}] 
		+ 304*theta[q,sqrtN]), {q,0,N1}]
Out[11]:= 1/q^3 - 248 q + 26 752 q^4 - 85 995 q^5 + 1 707 264 q^8 
		- 4 096 248 q^9 + 44 330 496 q^12 + ...
\end{lstlisting}

We drop the output of Listing 1 here, as the full pdf exported from Mathematica has $239$ pages. For the same reason, we omit the list of coefficients of $f_3(\tau)$ in Listing 2. Instead, we provide the pdf generated by Mathematica as well as a copyable input to sage as ancillary files with our submission to arXiv. We import the computed coefficients of $f_3(\tau)$ in SageMath \cite{sage}, and check for trivial congruences. We remark that our computational precision is up to the first $9$ coefficients of the logarithmic derivatives, because $10000/104 \in \big(9^2, 10^2\big)$. Consequently, we copy the outputs up to the first $9$ coefficients, and drop the higher ones as they are not meaningful.

\begin{lstlisting}[language=Sage, caption=Initialization of the logarithmic derivatives, captionpos=b]
sage: R.<q> = QQ[[]];
sage: f3 = # copy output of Listing 1 to here
sage: D=5
sage: F5=sum(sum(kronecker(-D,r)*n*f3[D*n^2]*q^(r*n) 
	for n in srange(100)) for r in srange(100))
\end{lstlisting}
The other logarithmic derivatives are initialized analogously to F5. We display their first $9$ coefficients modulo $11$.

\begin{lstlisting}[language=Sage, caption=Logarithmic derivatives modulo 11, captionpos=b]
sage: F5 % 11
3q + 5q^2 + 3q^3 + 6q^4 + 3q^5 + 5q^6 + 5q^9 + ...
sage: F20 % 11
4q + 4q^3 + 3q^4 + 4q^5 + 5q^8 + 3q^9 + ...
sage: F37 % 11
6q + 10q^2 + 4q^3 + q^4 + 6q^5 + 3q^6 + 9q^7 + q^9 + ...
sage: F53 % 11
3q + 5q^2 + 3q^3 + 6q^4 + 3q^5 + 5q^6 + 10q^7 + 5q^9 + ...
sage: F56 % 11
9q + 4q^2 + 9q^3 + 7q^4 + 9q^5 + 4q^6 + 4q^7 + 4q^9 + ...
sage: F80 % 11
7q^2 + 8q^4 + 7q^6 + 3q^8 + ...
sage: F89 % 11
q + 9q^2 + q^3 + 2q^4 + q^5 + 9q^6 + 9q^9 + ...
sage: F92 % 11
5q + q^2 + 5q^3 + 10q^4 + 5q^5 + q^6 + 2q^7 + q^9 + ...
sage: F97 % 11
7q + 8q^2 + q^3 + 3q^4 + 7q^5 + 9q^6 + 3q^9 + ...
sage: F104 % 11
5q + q^2 + 5q^3 + 10q^4 + 5q^5 + q^6 + q^9 + ...
\end{lstlisting}

From these expansions, we verify the congruences mentioned in Example \ref{exa:congexample}.
\begin{lstlisting}[language=Sage, caption=Some non-trivial congruences between logarithmic derivatives, captionpos=b]
sage: (7*(F5 - F53) + (F56 + 2*F89)) % 11
q^13 + 10*q^14 + ...
sage: (3*(F20 + 2*F80 + F97) + 2*(F37 + F92)) % 11
3*q^12 + 8*q^13 + 7*q^14 + ...
sage: (3*F5 + 2*F89) % 11
q^13 + 5*q^14 + ...
sage: ((3*F53 + 2*F89) - 2*(F56+2*F89)) % 11
10*q^13 + 7*q^14 + ...
sage: (2*(F104-F92) + (F56 + 2*F89)) % 11
2*q^13 + 9*q^14 + ...
sage: (2*(F5 - F53) + (F104-F92)) % 11
6*q^13 + 5*q^14 + ...
\end{lstlisting}
The vanishing of some of the coefficients corresponding to $q^{10}$, $q^{11}$ and $q^{12}$ modulo $11$ is a numerical coincidence, and not justified by our computational precision. The point is that we do indeed see the vanishing of the first $9$ coefficients modulo $11$ as desired. One may compute other non-trivial linear congruences modulo $11$ from this data as well.


\begin{thebibliography}{99}

\bibitem{Borcherds1} R. E. Borcherds, \emph{Automorphic forms on
$O_{s+2,2}(\R)$ and infinite products},
 Invent. Math. \textbf{120}  (1995), pages 161--213.

\bibitem{Borcherds2} R. E. Borcherds,
\emph{Automorphic forms with singularities on
Grassmannians}, Invent. Math. \textbf{132} (1998), pages 491--562.

\bibitem{Borcherds3} R. E. Borcherds, Reflection groups of Lorentzian lattices, Duke Math. J. {\bf 104} (2000), no.~2, pages 319--366.

\bibitem{AMSBook} K. Bringmann, A. Folsom, K. Ono, and L. Rolen,
\emph{Harmonic Maass forms and mock modular forms: Theory and applications}, Vol. 64, Amer. Math. Soc., Providence, 2017.

\bibitem{BruinierFunke} J. H. Bruinier and J. Funke, \emph{On two geometric theta lifts}, Duke Math. J. \textbf{125} (2004), pages 45--90.

\bibitem{BruinierOno} J. H. Bruinier and K. Ono, \emph{Heegner divisors, $L$-functions and
harmonic weak Maass forms}, Ann. of Math. \textbf{172} (2010), pages 2135--2181.

\bibitem{BruinierOno2} J. H. Bruinier\ and\ K. Ono, The arithmetic of Borcherds' exponents, Math. Ann. {\bf 327} (2003), no.~2, pages 293--303.

\bibitem{cox} D. Cox, \emph{Primes of the form $x^2 +ny^2$, Fermat, Class Field Theory, and Complex Multiplication},
Wiley Publ., New York, 1989.

\bibitem{deuring}  M. Deuring, Teilbarkeitseigenschaften der singul\"{a}ren Moduln der elliptischen Funktionen und die Diskriminante der Klassengleichung, Comment. Math. Helv. {\bf 19} (1946), pages 74--82.

\bibitem{gronts} M. Griffin, K. Ono\ and\ W.-L. Tsai, Heights of points on elliptic curves over $\Q$, Proc. Amer. Math. Soc. {\bf 149} (2021), no.~12, pages 5093--5100.

\bibitem{GKZ} B. Gross, W. Kohnen\ and\ D. Zagier, Heegner points and derivatives of $L$-series. II, Math. Ann. {\bf 278} (1987), no.~1-4, pages 497--562. 

\bibitem{IrelandRosen} K. Ireland and M. Rosen, \emph{A classical introduction to modern number theory}, Springer-Verlag, New York, 1990.

\bibitem{kanekozagier} M. Kaneko and D. Zagier, \emph{Supersingular j-invariants, hypergeometric series and Atkin’s orthogonal polynomials}, Comp. perspectives on Number Theory (Chicago, Illinois 1995), AMS/IP Stud. Adv. Math., Amer. Math. Soc. \textbf{7} (1998), pages 97-126.

\bibitem{Ono4} K. Ono, \emph{Parity of the partition function}, Adv.  Math. \textbf{225} (2010), pages 349-366.

\bibitem{Ono5} K. Ono, \emph{The partition function modulo 4}, in {\it Class groups of number fields and related topics}, Springer Proc. Math. Stat., {\bf470}, Springer, Singapore (2024), pages 67--76.

\bibitem{ono:cbms} K. Ono, \emph{The web of modularity : arithmetic of the coefficients of modular forms and q-series}, CBMS Regional Conference Series in Mathematics , {\bf102} (2003).

\bibitem{scheit} N. R. Scheithauer, The Weil representation of ${\rm SL}_2(\mathbb Z)$ and some applications, Int. Math. Res. Not. IMRN {\bf 2009}, no.~8, pages 1488--1545.

\bibitem{serre1}  J.-P. Serre, \emph{Congruences et formes modulaires (d'apr\'{e}s H.P.F. Swinnerton-Dyer)}, Sem. Bourbaki \textbf{416} (1971-1972), pages 74-88.

\bibitem{sage} W. A. Stein et al., \emph{Sage Mathematics Software} (Version 9.6), The Sage Development Team, 2022, \url{https://www.sagemath.org/}.

\bibitem{Sturm} J. Sturm, \emph{On the congruence of modular forms},
Springer Lect. Notes in Math. \textbf{1240}, Springer-Verlag, Berlin
(1984), pages 275-280.

\bibitem{wolfram} Wolfram Research, Inc., \emph{Mathematica}, Version 11.2, Champaign, IL (2017).

\bibitem{Zagiertraces} D. Zagier, \emph{Traces of singular moduli}, in {\it Motives, polylogarithms and Hodge theory, Part I (Irvine, CA, 1998)}, pages 211--244, Int. Press Lect. Ser., 3, I, Int. Press, Somerville, MA.

\end{thebibliography}
\end{document}